\documentclass[final,leqno]{siamltex704}
\usepackage{hyperref}
\usepackage{amsmath}
\usepackage[notcite,notref]{showkeys}
\usepackage{amsfonts,amssymb}
\usepackage{dsfont}
\usepackage{pifont}

\newtheorem{remark}{Remark}[section]

\newtheorem{algorithm}{Weak Galerkin Algorithm}
\newtheorem{model-problem}{Problem}

\newcommand{\bu}{\textbf{u}}
\newcommand{\bw}{\textbf{w}}

\newcommand{\bx}{\textbf{x}}

\newcommand{\be}{\textbf{e}}
\newcommand{\bv}{\textbf{v}}

\newcommand{\bH}{{\bf H}}
\newcommand{\curl}{{\nabla\times}}

\def\T{{\mathcal T}}
\def\E{{\mathcal E}}

\def\pT{{\partial T}}

\def\jump#1{{[\![#1[\!]}}

\def\bn{{\bf n}}

\def\bV{{\mathbf V}}

\def\bH{{\mathbf H}}

\def\3bar{{|\hspace{-.02in}|\hspace{-.02in}|}}

\def\leqC{\lesssim}
\def\Hmudiv#1{{H({\rm div}_ \varepsilon; #1)}}
\def\Hzmudiv#1{{H_0({\rm div}_ \varepsilon; #1)}}

\def\bE{\textbf{E}}
\def\bH{\textbf{H}}

\begin{document}

\setlength{\parindent}{0.25in} \setlength{\parskip}{0.08in}

\title{New Discretization Schemes for Time-Harmonic Maxwell Equations by Weak
Galerkin Finite Element Methods}

\author{
Chunmei Wang\thanks{Department of Mathematics, Texas State
University, San Marcos, TX 78666, USA. The research of Chunmei Wang
was partially supported by National Science Foundation Award
DMS-1522586, National Natural Science Foundation of China Award
\#11526113, Jiangsu Key Lab for NSLSCS Grant \#201602, and by
Jiangsu Provincial Foundation Award \#BK20050538.}}

\maketitle

\begin{abstract}
This paper introduces new discretization schemes for time-harmonic
Maxwell equations in a connected domain by using the weak Galerkin
(WG) finite element method. The corresponding WG algorithms are
analyzed for their stability and convergence. Error estimates of
optimal order in various discrete Sobolev norms are established for
the resulting finite element approximations.

\end{abstract}



\section{Introduction}
This paper is concerned with new developments of numerical methods
for time-harmonic Maxwell equations. The time-harmonic Maxwell
equations are coupled magnetic and electric equations given by
\begin{equation}\label{maxeqn}
 \begin{split}
 \nabla\times \textbf{E} =&- \frac{\partial \textbf{B}}{\partial t},\qquad \quad \ \text{ in}\ \Omega,\\
\nabla \times  \textbf{H}=& \frac{\partial \textbf{D}}{\partial t}+ \textbf{j},\qquad  \quad \text{ in}\ \Omega,\\
\nabla\cdot\textbf{ D}=&\rho,\qquad \qquad \qquad \text{in}\ \Omega,\\
\nabla\cdot\textbf{ B} =&0,\qquad \qquad \qquad \text{in}\ \Omega,
 \end{split}
\end{equation}
with the constitutive relations:
$$
\textbf{ B}=  \mu \textbf{ H}, \textbf{ j}=\sigma \textbf{ E}+\textbf{ j}_e, \textbf{ D}=\varepsilon\textbf{E},
$$
where $\Omega$ is an open bounded and connected domain in
$\mathbb{R}^d (d=2, 3) $ with a Lipschitz continuous boundary
$\Gamma=\partial\Omega$. Here, $\textbf{ E}$  is the electric field
intensity, $\textbf{ B}$ is the magnetic flux density, $\textbf{H}$
is the magnetic field intensity, $\textbf{ D}$ is the electric
displacement flux density, $\textbf{ j}$ is the electric current
density, $\mu=\{\mu_{ij}(\bx)\}_{d\times d}$ is called permeability,
$\rho$ is the charge density, $\textbf{j}_e$ is the external current
density, $\sigma$ is real-valued and is known as the electric
conductivity, and $\varepsilon=\{\varepsilon_{ij}(\bx)\}_{d\times
d}$ is the material parameter, and is called permittivity.
Additionally, $\mu$, $\varepsilon$ are real-valued, symmetric,
uniformly positive definite matrices in the domain $\Omega$. We
assume that $\mu$, $\varepsilon$ and $\sigma$ are piecewise smooth
functions in the domain $\Omega$.

For time-harmonic fields, where the time dependence is assumed to be
harmonic, i.e., $\text{exp}(i \omega t)$, using the constitutive
relations, the maxwell equations (\ref{maxeqn}) can be rewritten for
the Fourier transform of the fields as (see \cite{cai2013} for
details)
\begin{eqnarray}\label{time-harmonic-ME-01}
 \nabla\times \textbf{E} & = & - i\omega  \mu \textbf{H},\qquad\qquad\quad  \text{ in}\ \Omega,\\
 \nabla\times \textbf{H}& = & i\omega \varepsilon \textbf{E} +\sigma \textbf{E} +\textbf{j}_e,\
 \quad \text{ in}\ \Omega,\label{time-harmonic-ME-02}\\
\nabla\cdot (\varepsilon \textbf{E}) & = & \rho,\qquad\qquad\qquad\qquad \text{in}\ \Omega,\label{time-harmonic-ME-03}\\
\nabla\cdot ( \mu\textbf{H}) & = & 0,\qquad\qquad\qquad\qquad \text{in}\ \Omega,\label{time-harmonic-ME-04}
\end{eqnarray}
where $\omega $ is a constant in the domain $\Omega$.

In the past several decades, the Maxwell equations have been
extensively investigated by many researchers. H(curl) conforming
finite element method was first  introduced by J. N\'{e}d\'{e}lec
 \cite{n1980} and was further developed by P. Monk
\cite{m2003}. Houston, Perugia and Schotzau \cite{hps2002,
hps2003, hps2004, ps2003, psm2002} have developed
  discontinuous Galerkin (DG) finite element methods for
  the Maxwell  equations.  Particularly in
\cite{hps2004}, a mixed DG formulation for the Maxwell equations
was introduced and analyzed.   Recently, a
weakly over-penalized symmetric interior penalty method
\cite{bls2008} has been introduced and analyzed by S. Brenner, F.
Li and L. Sung. There are also many other numerical methods developed to
discretize the Maxwell equations.

Recently, WG method is emerging as an efficient finite element
technique for partial differential equations. The WG finite element
method was first introduced in \cite{wy1202, wy2013} for second
order elliptic equations  and the idea was subsequently further
developed for several other model PDEs \cite{ mwy1204, cwang1,
cwang2, cwang3, cwang4, wy1302}. The key idea of WG method is to use
weak functions and their corresponding discrete weak derivatives in
existing variational forms. WG method is highly flexible and robust
by allowing the use of discontinuous piecewise polynomials and
finite element partitions with arbitrary shape of
polygons/polyhedra, and the method is parameter free and absolutely
stable. WG finite element method has been applied to time-harmonic
Maxwell equations in \cite{mwyz}, yielding a numerical method that
has optimal order of convergence in certain discrete norms.

The goal of this paper is to present a new WG finite element method
for the time-harmonic Maxwell equations
(\ref{time-harmonic-ME-01})-(\ref{time-harmonic-ME-04}) in a
connected domain with heterogeneous media, which covers more cases
compared with the model problem considered in \cite{mwyz}. In
particular, we formulate the time-harmonic Maxwell equations
(\ref{time-harmonic-ME-01})-(\ref{time-harmonic-ME-04}) into two
variational problems with complex coefficients; see (\ref{maxwell1})
and (\ref{maxwell2}) for details. Each of the variational problems
is then discretized by using the weak Galerkin finite element
method. The main difficulty in the design of numerical methods for
(\ref{maxwell1}) and (\ref{maxwell2}) lies in the fact that the
terms $\nabla\cdot( \varepsilon \textbf{E})$ and $\nabla\cdot(\mu
\textbf{H})$ require the continuity of $\varepsilon \textbf{E}$ and
$\mu \textbf{H}$ in the normal direction of all interior interfaces,
respectively. Consequently, the usual $H(div)$ or $H(curl)$
conforming elements are not applicable in this practice. This paper
shows that the weak Galerkin finite element method offers an ideal
solution, as the continuity can be relaxed by a weak continuity
implemented through a carefully chosen stabilizer.

The paper is organized as follows. In Section \ref{Section:sec3}, we
shall derive two variational problems: one for the electric field
intensity and the other for the magnetic field intensity. These
variational problems form the basis of the weak Galerkin finite
element methods of this paper. In Section \ref{Section:sec4}, we
shall briefly review the discrete weak divergence and the discrete
weak curl operators which are necessary in weak Galerkin. In Section
\ref{Section:sec5}, we describe how the weak Galerkin finite element
algorithms are formulated. Section \ref{Section:sec6} is devoted to
a verification of some stability conditions for the resulting WG
algorithms. In particular, it is shown in this section that the WG
algorithms have one and only one solution. In Section
\ref{Section:sec7}, we derive some error equations for our WG
algorithms. Finally in Section \ref{Section:sec8}, we establish some
optimal order error estimates for the WG finite element
approximations.

Throughout the paper, we will follow the usual notations for Sobolev
spaces and norms \cite{ciarlet}. For any open bounded domain
$D\subset \mathbb{R}^d (d=2, 3)$ with Lipschitz continuous boundary, we use
$\|\cdot\|_{s,D}$ and $|\cdot|_{s,D}$ to denote the norm and
seminorm in the Sobolev space $H^s(D)$ for any $s\ge 0$,
respectively. The inner product in $H^s(D)$ is denoted by
$(\cdot,\cdot)_{s,D}$. The space $H^0(D)$ coincides with $L^2(D)$,
for which the norm and the inner product are denoted by $\|\cdot
\|_{D}$ and $(\cdot,\cdot)_{D}$, respectively.

We introduce the following Sobolev space
\begin{equation*}
\Hmudiv{D}=\{\bv\in [L^2(D)]^d: \ \nabla\cdot( \varepsilon\bv)\in L^2(D)\},
\end{equation*}
with norm given by
$$
\|\bv\|_{\Hmudiv{D}} = (\|\bv\|^2_D+\|\nabla\cdot ( \varepsilon
\bv)\|^2_D)^{\frac12},
$$
where $\nabla\cdot( \varepsilon\bv)$ is the divergence of $ \varepsilon\bv$. Any
$\bv\in \Hmudiv{D}$ can be assigned a trace for the normal component
of $ \varepsilon\bv$ on the boundary. Denote the subspace of $\Hmudiv{D}$  with vanishing trace in
the normal component  by
$$
\Hzmudiv{D}=\{\bv\in \Hmudiv{D}: \ ( \varepsilon\bv)\cdot\bn|_{\partial D} =
0\}.
$$
When $ \varepsilon=I$ is the identity matrix, the spaces
$\Hmudiv{D}$ and $\Hzmudiv{D}$ are denoted as $H({\rm div}; D)$ and
$H_0({\rm div}; D)$, respectively.

We also use the following Sobolev space
$$
H({\rm curl}; D)=\{\bv: \bv\in [L^2(D)]^d, \nabla \times \bv\in
[L^2(D)]^d\}
$$
with norm given by
$$
\|\bv\|_{H({\rm curl}; D)}=(\|\bv\|^2_D+\|\nabla\times
\bv\|^2_D)^{\frac12},
$$
where $\nabla\times\bv$ is the curl of $\bv$. Any $\bv\in H({\rm
curl}; D)$ can be assigned a trace for its tangential component on
the boundary. Denote the subspace of $H({\rm curl}; D)$ with vanishing
trace in the tangential component   by
$$
H_0({\rm curl}; D)=\{\bv\in H({\rm curl}; D): \
\bv\times\bn|_{\partial D} = 0 \}.
$$

When $D=\Omega$, we shall drop the subscript $D$ in the norm and
inner product notation. For convenience, throughout the paper, we
use ``$\lesssim$ '' to denote ``less than or equal to up to a
general constant independent of the mesh size or functions appearing
in the inequality".

\section{Variational Formulations}\label{Section:sec3}

The goal of this section is to derive two different variational
formulations for the time-harmonic Maxwell model problem
(\ref{time-harmonic-ME-01})-(\ref{time-harmonic-ME-04}).

\subsection{Variational Formulation I}
For the electric field intensity $\textbf{E}$, we first apply the
differential operator $\nabla\times\mu^{-1}$ to
(\ref{time-harmonic-ME-01}), and then use the equation
(\ref{time-harmonic-ME-02}) to obtain
\begin{equation}\label{Maxwell-Equations-New}
 \nabla\times ( \mu^{-1} \nabla\times \textbf{E}) = (\omega^2 \varepsilon - i\omega\sigma)
\textbf{E}- i\omega \textbf{j}_e, \qquad\text{in}\ \Omega.
\end{equation}
A typical boundary condition for the electric field intensity
$\textbf{E}$ is given by
\begin{equation}\label{bond}
\textbf{E}\times\textbf{n} = 0, \qquad \mbox{ on }  \Gamma,
\end{equation}
where $\bn$ is the unit outward normal direction to $\Gamma$.

Therefore, a variational formulation for the electric field
intensity $\textbf{E}$ seeks $\textbf{E}\in  H_0(\mbox{curl};
\Omega) \cap H({\text{div}}_\varepsilon;\Omega)$ and $p\in
L^2(\Omega)$ such that
\begin{equation}\label{maxwell1}
 \begin{split}
 ( \mu^{-1}\nabla\times\textbf{E},\
\nabla\times\textbf{v})+((i\omega\sigma-\omega^2\varepsilon)\textbf{E},
\bv)-(\nabla\cdot(\varepsilon \textbf{v}), p)&=  -(i\omega \textbf{j}_e,\textbf{v}),\\
(\nabla\cdot(\varepsilon\textbf{E}), q)&= (\rho,q),
 \end{split}
\end{equation}
for all $\textbf{v} \in  H_0(\mbox{curl}, \Omega)\cap
H(\mbox{div}_\varepsilon, \Omega) $ and $q\in L^2(\Omega)$.

\subsection{Variational Formulation II}
For the magnetic field intensity $\textbf{H}$, we apply $
\nabla\times (i\omega\varepsilon+\sigma)^{-1}$ to the equation
(\ref{time-harmonic-ME-02}), and then use the equation
(\ref{time-harmonic-ME-01}) to obtain
\begin{equation*}\label{Maxwell-Equations-NewNew}
 \nabla\times ((i\omega\varepsilon+\sigma)^{-1} \nabla\times \textbf{H})=-i\omega \mu\textbf{H} +
 \nabla\times(i\omega\varepsilon+\sigma)^{-1} \textbf{j}_e, \qquad\text{in}\ \Omega.
\end{equation*}
  The boundary conditions are
\begin{equation*}
 \begin{split}
((i\omega\varepsilon+\sigma)^{-1} \nabla\times \textbf{H}) \times \textbf{n} = 0, \qquad\text{on}\  \Gamma,\\
 \mu \textbf{H}\cdot \textbf{n} = 0, \qquad\text{on}\  \Gamma.
 \end{split}
\end{equation*}
 Note that $\textbf{j}_e$ as a volume current has no contribution on the boundary $\Gamma$.

A variational formulation for the magnetic field intensity
$\textbf{H}$ seeks $\textbf{H}\in
 H(\mbox{curl}; \Omega) \cap H_0(\mbox{div}_\mu; \Omega) $ and $p\in L^2_0(\Omega)$ such that
\begin{equation}\label{maxwell2}
 \begin{split}
 ((i\omega\varepsilon+\sigma)^{-1}\nabla\times\textbf{H},\
\nabla\times\textbf{v})+(i\omega \mu\textbf{H}, \textbf{v})
-(\nabla\cdot( \mu \textbf{v}), p)&=
(\nabla\times (i\omega\varepsilon+\sigma)^{-1}  \textbf{j}_e,\bv),  \\
(\nabla\cdot( \mu \textbf{H}), q)& = 0,
 \end{split}
\end{equation}
 for all $ \textbf{v}\in  H(\hbox{curl};\Omega)\cap H_0(\mbox{div}_ \mu;
\Omega) $ and $q\in L^2_0(\Omega)$.

For simplicity, throughout the paper, we assume that $\mu$, $\sigma$
and $\varepsilon$ are piecewise constants in the domain $\Omega$
with respect to the finite element partitions to be specified in
forthcoming sections. The results can be extended to piecewise
smooth coefficients without any technical  difficulties.

\section{Weak Differential Operators}\label{Section:sec4}
The variational formulations (\ref{maxwell1}) and (\ref{maxwell2})
are based on two differential operators: divergence and curl. In
this section, we will introduce weak divergence operator for
vector-valued functions of the form $\varepsilon\bv$ and then review
the definition for the weak curl operator. More details can be found
in \cite{cwang1}.

Let $K\subset\Omega$ be any open bounded domain with boundary
$\partial K$. Denote by $\bn$ the unit outward normal direction on
$\partial K$. The space of weak vector-valued functions in $K$ is
defined as follows
\begin{eqnarray*}
V(K)=\{\bv=\{\bv_0,\bv_{b}\}: \ \textbf{v}_0\in [L^2(K)]^d,\
\bv_{b}\in [L^2(\partial K)]^d\},
\end{eqnarray*}
where $\textbf{v}_0$ represents the value of $\textbf{v}$ in the
interior of $K$, and $\bv_{b}$ the information of $\bv$ on the
boundary $\partial K$. There are two piece of information of $\bv$
on $\partial K$ which are needed in the variational formulations
(\ref{maxwell1}) and (\ref{maxwell2}): one of them is the tangential
component $\bn\times(\bv \times\bn)$ and the other one is the normal
component of $ \varepsilon\bv$ on $\partial K$ given by $(
\varepsilon\bv\cdot\bn)\bn$. Intuitively, the vector $\bv_b$ is used
to represent both of them as follows
\begin{equation}\label{interteller.00}
\bv_b =( \varepsilon\bv\cdot\bn)\bn +\bn\times(\bv\times \bn).
\end{equation}
We emphasize that the right-hand side of (\ref{interteller.00}) is
not meant to be a decomposition of the trace of $\bv$ on $\partial
K$.

\subsection{Weak divergence and discrete weak divergence \cite{cwang1,
wy1302}} For any $\textbf{v}\in V(K)$, the weak divergence of $
\varepsilon\bv$, denoted by $\nabla_{w,K}\cdot (
\varepsilon\textbf{v})$, is defined as a bounded linear functional
on the Sobolev space $H^1(K)$ satisfying
\begin{equation*}
\langle \nabla_{w,K}\cdot
( \varepsilon\textbf{v}),\varphi\rangle_K=-( \varepsilon\textbf{v}_0,\nabla
\varphi)_K+\langle \bv_{b}\cdot\bn,\varphi\rangle_{\partial
K},\qquad \forall \ \varphi\in H^1(K).
\end{equation*}
Here the left-hand side stands for the action of the linear
functional on $\varphi\in H^1(K)$, and $\langle\cdot,
\cdot\rangle_{\partial K}$ is the inner product in $L^2(\partial
K)$. The discrete weak divergence of $ \varepsilon\bv$, denoted by
$\nabla_{w,r,K}\cdot ( \varepsilon\textbf{v})$, is defined as the unique
polynomial in $P_r(K)$, $r\ge 0$, satisfying
\begin{equation}\label{2.2}
(\nabla_{w,r,K}\cdot
( \varepsilon\textbf{v}),\varphi)_K=-( \varepsilon\textbf{v}_0,\nabla
\varphi)_K+\langle {\bv}_{b}\cdot\bn, \varphi\rangle_{\partial K},
\quad\forall\ \varphi\in P_r(K),
\end{equation}
where $P_r(K)$ is the set of all polynomials on $K$ with degree $r$
or less.

Assume that $\bv_0$ is sufficiently smooth such that
$\nabla\cdot( \varepsilon\bv_0)\in L^2(K)$. By applying the integration by
parts to the first term on the right-hand side of (\ref{2.2}), we
have
\begin{equation}\label{2.2new}
\begin{split}
(\nabla_{w,r,K} \cdot( \varepsilon\bv),\varphi)_K =(\nabla\cdot( \varepsilon\bv_0),
\varphi)_K+\langle (\bv_{b}- \varepsilon\bv_0)\cdot\bn,
\varphi\rangle_{\partial K},
\end{split}
\end{equation}
for any $\varphi\in P_r(K)$.

\subsection{Weak curl and discrete weak curl \cite{mwyz, cwang1}}
The weak curl of $\textbf{v} \in V(K)$, denoted by $\nabla_{w,K}
\times \textbf{v}$, is defined as a bounded linear functional on the
Sobolev space $[H^1(K)]^d$ satisfying
$$
\langle \nabla_{w,K} \times \textbf{v},\varphi\rangle_K=(\textbf{v}_0,
\nabla\times \varphi)_K-\langle \textbf{v}_{b}\times
\textbf{n},\varphi\rangle_{\partial K},\quad\ \forall \ \varphi\in
[H^1(K)]^d.
$$

The discrete weak curl of $\textbf{v} \in V(K)$, denoted by
$\nabla_{w,r,K} \times \textbf{v}$, is defined as the unique
polynomial-valued vector in $[P_r (K)]^d$, such that
\begin{equation}\label{2.5}
(\nabla_{w,r,K} \times \textbf{v},\varphi)_K=(\textbf{v}_0,
\nabla\times \varphi)_K-\langle \textbf{v}_{b}\times
\textbf{n},\varphi\rangle_{\partial K}, \quad\forall\varphi\in [P_r
(K)]^d.
\end{equation}

For sufficiently smooth $\bv_0$ with $\nabla\times\bv_0\in
[L^2(K)]^d$, by applying the integration by parts to the first term
on the right-hand side of (\ref{2.5}), we obtain
\begin{equation}\label{2.5new}
\begin{split}
(\nabla_{w,r,K} \times \bv,\varphi)_K =(\nabla\times\bv_0,
\varphi)_K-\langle (\bv_{b}-\bv_0)\times\bn,
\varphi\rangle_{\partial K},
\end{split}
\end{equation}
for any $\varphi\in [P_r(K)]^d$.

\begin{remark}
All the definitions and formulations with respect to the coefficient
$\varepsilon$ of this section can be generalized to the coefficient
$\mu$. This is particularly useful in the study of the equation for
the magnetic field intensity function.
\end{remark}

\section{Numerical Algorithms by Weak Galerkin}\label{Section:sec5}

Let ${\cal T}_h$ be a finite element partition of the domain
$\Omega\subset \mathbb{R}^d (d=2, 3)$ with mesh size $h$. Assume
that $\T_h$ consists of polygons/polyhedra of arbitrary shape and is
shape regular as defined in \cite{wy1202}. Denote by ${\cal E}_h$
the set of all edges/faces in ${\cal T}_h$ and ${\cal E}^0_h={\cal
E}_h\setminus {\partial\Omega}$ the set of all interior edges/faces
in ${\cal T}_h$. For each interior edge/face $e\in {\cal E}^0_h$, we
assign a prescribed normal direction $\bn_e$ to $e$. Denote by $\bn$
the unit outward normal direction to the boundary $\Gamma$. Denote
the jump of $q$ on the edge/face $e\in\E_h$ by
\begin{equation}\label{jump}
 \jump{q}=\left\{    \begin{array}{cc}
  q|_{\partial T_1}-  q|_{\partial T_2},&e\in\E_h^0, \\
   q, & e  \subset \partial \Omega,\\
  \end{array}\right.
\end{equation}
where $q|_{\partial T_i}$ denotes the value of $q$ on an edge/face
$e$ as seen from the element $T_i$, $i=1, 2$. Here $T_1$ and $T_2$
are the two elements that share $e$ as a common edge/face. The order
of $T_1$ and $T_2$ is non-essential in (\ref{jump}) as long as the
difference is taken in a consistent way in all the formulas. If
$e\subset\Gamma$ is a boundary edge, then $\jump{q} = q|_e$ is
defined as its trace on $e$.

Let $k\ge 1$ be a given integer. For each element $T\in\T_h$, we
define the local weak  finite element space by
$$
\bV(k,T)=\{ \bv=\{\bv_0,\bv_b\}: \bv_0\in [P_k(T)]^d, \bv_b\in
[P_{k}(e)]^d,\ e\in (\partial T\cap\E_h)\}.
$$
By patching the local weak finite element space $\bV(k,T)$ together
over all the elements $T\in {\cal T}_h$ through a common value
$\bv_b$ on the interior interface $\E_h^0$, we obtain a global weak
finite element space:
\begin{equation}\label{EQ:global-WFES}
\bV_h=\{\bv=\{\bv_0,\bv_b\}:\; \bv|_T\in \bV(k,T), \ T\in \T_h \}.
\end{equation}

Introduce two subspaces of $\bV_h$ as follows:
 \begin{equation*}\label{EQ:global-WFES-testspace-problem2}
\bV_{h,0}^1=\{\bv=\{\bv_0,\bv_b\}\in\bV_h:\; \bv_b  \times\bn = 0 \
\mbox{on} \ \Gamma\},
\end{equation*}
 \begin{equation*}\label{EQ:global-WFES-testspace-problem1}
\bV_{h,0}^2 =\{\bv=\{\bv_0,\bv_b\}\in\bV_h:\; \bv_b  \cdot\bn = 0 \
\mbox{on} \ \Gamma\}.
\end{equation*}
We further define two more finite element spaces
$$
W_h^1 =\{q: \ q\in L^2(\Omega), \ q|_T\in P_{k-1}(T), T\in\T_h\},
$$
$$
W_h^2 =\{q: \ q\in L_0^2(\Omega), \ q|_T\in P_{k-1}(T), T\in\T_h\}.
$$

The discrete weak divergence $(\nabla_{w,k-1}\cdot ) $ and the
discrete weak curl $(\nabla_{w,k-1} \times)$ can be computed by
using (\ref{2.2}) and (\ref{2.5}) on each element; i.e.,
\begin{align*}
(\nabla_{w,k-1}\cdot ( \varepsilon\textbf{v}))|_T=&\nabla_{w,k-1,T}\cdot
( \varepsilon\textbf{v}|_T), \quad \ \textbf{v}\in \bV_h,\\
(\nabla_{w,k-1}\cdot ( \mu\textbf{v}))|_T=&\nabla_{w,k-1,T}\cdot
( \mu\textbf{v}|_T), \quad \ \textbf{v}\in \bV_h,\\
(\nabla_{w,k-1}\times  \textbf{v})|_T=&\nabla_{w,k-1,T}\times
(\textbf{v}|_T), \qquad \textbf{v}\in \bV_h.
\end{align*}
For simplicity of notation and without confusion, we shall drop the subscript $k-1$ from
the notations $(\nabla_{w,k-1}\cdot )$ and
$(\nabla_{w,k-1}\times)$ from now on.

We introduce the following bilinear forms
\begin{align*}
a_1(\textbf{v},\textbf{w})=&\sum_{T\in {\cal T}_h}(\mu^{-1}\nabla_{w}\times
\textbf{v},\nabla_{w} \times
\textbf{w})_T+  ((i\omega \sigma-\omega^2 \varepsilon)
\textbf{v}_0,\textbf{w}_0)_T+s_1(\textbf{v},\textbf{w}),\\
a_2(\textbf{v},\textbf{w})=&\sum_{T\in {\cal T}_h}((i \omega \varepsilon+ \sigma)^{-1}\nabla_{w}\times
\textbf{v},\nabla_{w} \times
\textbf{w})_T+    (i\omega \mu \textbf{v}_0,\textbf{w}_0)_T+s_2(\textbf{v},\textbf{w}),\\
b_1(\textbf{v},q)=& \sum_{T\in {\cal T}_h}(\nabla_{w}\cdot ( \varepsilon\textbf{v}),
q)_T, \\
b_2(\textbf{v},q)=&\sum_{T\in {\cal T}_h}(\nabla_{w}\cdot ( \mu\textbf{v}),q)_T,
\end{align*}
where
\begin{align*}
 s_1(\textbf{v},\textbf{w})=&\sum_{T\in {\cal
T}_h}h_T^{-1}\langle
 (  \varepsilon\textbf{v}_0-\textbf{v}_b)\cdot\bn,
(\varepsilon \textbf{w}_0-\textbf{w}_b)\cdot\bn\rangle_{\partial T}  +h_T^{-1}\langle
 (\textbf{v}_0-\textbf{v}_b)\times\bn,
(\textbf{w}_0-\textbf{w}_b)\times\bn\rangle_{\partial T},\\
 s_2(\textbf{v},\textbf{w})=&\sum_{T\in {\cal
T}_h}h_T^{-1}\langle
 ( \mu\textbf{v}_0-\textbf{v}_b)\cdot\bn,
( \mu\textbf{w}_0-\textbf{w}_b)\cdot\bn\rangle_{\partial T} +h_T^{-1}\langle
 (\textbf{v}_0-\textbf{v}_b)\times\bn,
(\textbf{w}_0-\textbf{w}_b)\times\bn\rangle_{\partial T}.
\end{align*}

\begin{algorithm}\label{algo1} For a numerical approximation of the electric field intensity $\bE$, one may
seek $\bE_h \in \bV_{h,0}^1$ and an auxiliary function $p_h\in
W^1_h$, such that
\begin{align}\label{3.3}
a_1(\textbf{E}_h,\textbf{v})-b_1(\textbf{v},p_h)=& -(i \omega
\textbf{j}_e,\textbf{v}_0), \qquad \forall
\textbf{v}=\{\textbf{v}_0, \textbf{v}_b\}\in \bV_{h,0}^1,
\\
b_1(\textbf{E}_h,w) =&( \rho,w) , \qquad\qquad\quad \forall \ w\in
W^1_h.\label{3.4}
\end{align}
\end{algorithm}

\begin{algorithm}\label{algo2} For a numerical approximation of the magnetic field intensity $\bH$, one may
seek $\textbf{H}_h  \in \bV_{h,0}^2$ and an auxiliary function
$p_h\in W^2_h$, such that
\begin{align}\label{3.5}
a_2(\textbf{H}_h,\textbf{v})-b_2(\textbf{v},p_h)=& ( \nabla \times
(i \omega \varepsilon +\sigma)^{-1} \textbf{j}_e,\textbf{v}_0),
\qquad \forall  \textbf{v} =\{\textbf{v}_0, \textbf{v}_b\}\in
\bV_{h,0}^2,
\\
b_2(\textbf{H}_h,w) =&0 , \qquad\qquad\qquad\qquad\qquad\quad
\forall \ w\in W^2_h.\label{3.6}
\end{align}
\end{algorithm}

\section{Verification of Stability Conditions}\label{Section:sec6}
We first introduce two norms: one in the weak finite element space
 $\bV_{h, 0}^1$ and the other in $\bV_{h,  0}^2$ as follows:
\begin{equation}\label{Eq:bar-norm}
\begin{split}
&\3bar\textbf{v}\3bar_{\bV_{h, 0}^1}  = \Big(\sum_{T\in {\cal T}_h}
\|\nabla_{w}\times \textbf{v}\|_T^2+  \|\bv_0\|_T^2 \\
& + h_T^{-1}\|
 \varepsilon  \textbf{v}_0\cdot\bn-\textbf{v}_b\cdot\bn\|^2_{\partial T}
+ h_T^{-1}\|
\textbf{v}_0\times\bn-\textbf{v}_b\times\bn\|^2_{\partial T}
 \Big)^{\frac{1}{2}},\quad \forall  \textbf{v} \in \bV_{h, 0}^1, \\
 \end{split}
\end{equation}
\begin{equation}\label{Eq:bar-norm-2}
\begin{split}
& \3bar\textbf{v}\3bar_{\bV_{h, 0}^2}   =\Big(\sum_{T\in {\cal T}_h}
\|\nabla_{w}\times \textbf{v}\|_T^2+  \|\bv_0\|_T^2 \\
& + h_T^{-1}\|
  \mu\textbf{v}_0\cdot\bn-\textbf{v}_b\cdot\bn\|^2_{\partial T}
+ h_T^{-1}\|
 \textbf{v}_0\times\bn-\textbf{v}_b\times\bn\|^2_{\partial T}
 \Big)^{\frac{1}{2}}, \quad  \forall  \textbf{v} \in \bV_{h, 0}^2.
 \end{split}
\end{equation}

In the finite element spaces $W_h^1$ and $W_h^2$, we introduce
mesh-dependent norms as follows
\begin{equation}\label{Eq:bar-norm-q1}
\begin{split}
 \|q\|_{W_h^1} =  \Big(h^2\sum_{T\in {\cal T}_h} ( \varepsilon\nabla q,\nabla q)_T +
h\sum_{e\in \E_h} \|\jump{q}\|_e^2 \Big)^{\frac{1}{2}}, \qquad
\forall q\in W_h^1,
 \end{split}
\end{equation}
\begin{equation}\label{Eq:bar-norm-q2}
\begin{split}
  \|q\|_{W_h^2} =  \Big(h^2\sum_{T\in {\cal T}_h} ( \mu\nabla q,\nabla q)_T +
h\sum_{e\in \E_h^0} \|\jump{q}\|_e^2\end{split} \Big)^{\frac{1}{2}},\qquad \forall q\in W_h^2.
\end{equation}

The following two lemmas are concerned with the coercivity of the
bilinear forms $a_1(\cdot,\cdot)$ and $a_2(\cdot,\cdot)$. The
boundedness of these two bilinear forms is straightforward.

\begin{lemma} There exists a positive constant $C$ such that for any $\bv\in \bV_{h,
0}^1$ one has
\begin{equation}\label{EQ:aone-coercivity}
|a_1(\bv,\bv)| \ge C \3bar\textbf{v}\3bar_{\bV_{h, 0}^1}^2.
\end{equation}
\end{lemma}

\begin{proof} From the definition of the bilinear form
$a_1(\cdot,\cdot)$ we have
$$
a_1(\textbf{v},\textbf{v})=\sum_{T\in {\cal
T}_h}(\mu^{-1}\nabla_{w}\times \textbf{v},\nabla_{w} \times
\textbf{v})_T+  ((i\omega \sigma-\omega^2 \varepsilon)
\textbf{v}_0,\textbf{v}_0)_T+s_1(\textbf{v},\textbf{v}).
$$
Since imaginary part of $a_1(\textbf{v},\textbf{v})$ is given by
$(\omega\sigma \bv_0, \bv_0)$, then we have
\begin{equation}\label{EQ:10.11.2016:001}
\omega \sigma_0 \|\bv_0\|^2 \leq (\omega\sigma \bv_0, \bv_0) \leq
|a_1(\textbf{v},\textbf{v})|,
\end{equation}
where $\sigma_0$ is the minimum value of $\sigma$ over $\Omega$. The
real part of $a_1(\textbf{v},\textbf{v})$ is given by
$$
Re(a_1(\textbf{v},\textbf{v})) = \sum_{T\in {\cal
T}_h}(\mu^{-1}\nabla_{w}\times \textbf{v},\nabla_{w} \times
\textbf{v})_T -  (\omega^2 \varepsilon
\textbf{v}_0,\textbf{v}_0)_T+s_1(\textbf{v},\textbf{v}).
$$
Thus,
\begin{equation}\label{EQ:10.11.2016:002}
\left| \sum_{T\in {\cal T}_h}(\mu^{-1}\nabla_{w}\times
\textbf{v},\nabla_{w} \times \textbf{v})_T -  (\omega^2 \varepsilon
\textbf{v}_0,\textbf{v}_0)_T+s_1(\textbf{v},\textbf{v}) \right| \leq
|a_1(\bv,\bv)|.
\end{equation}
Combining (\ref{EQ:10.11.2016:001}) with (\ref{EQ:10.11.2016:002})
gives rise to the coercivity estimate (\ref{EQ:aone-coercivity}).
This completes the proof of the lemma.
\end{proof}

\begin{lemma} There exists a positive constant $C$ such that for any $\bv\in \bV_{h,
0}^2$ one has
\begin{equation}\label{EQ:atwo-coercivity}
|a_2(\bv,\bv)| \ge C \3bar\textbf{v}\3bar_{\bV_{h, 0}^2}^2.
\end{equation}
\end{lemma}

\begin{proof} From the definition of the bilinear form
$a_2(\cdot,\cdot)$ we have
$$
a_2(\textbf{v},\textbf{v})= \sum_{T\in {\cal T}_h}((i \omega
\varepsilon+ \sigma)^{-1}\nabla_{w}\times \textbf{v},\nabla_{w}
\times \textbf{v})_T+ (i\omega \mu
\textbf{v}_0,\textbf{v}_0)_T+s_2(\textbf{v},\textbf{v})
$$
The real part of $a_2(\textbf{v},\textbf{v})$ is given by
$$
Re(a_2(\textbf{v},\textbf{v})) = \sum_{T\in {\cal
T}_h}(\sigma(\sigma^2 + \omega^2 \varepsilon^2)^{-1}\nabla_{w}\times
\textbf{v},\nabla_{w} \times \textbf{v})_T+
s_2(\textbf{v},\textbf{v})\ge 0.
$$
Hence, we have
\begin{equation}\label{EQ:10.11.2016:010}
\sum_{T\in {\cal T}_h}(\sigma(\sigma^2 + \omega^2
\varepsilon^2)^{-1}\nabla_{w}\times \textbf{v},\nabla_{w} \times
\textbf{v})_T+ s_2(\textbf{v},\textbf{v}) \leq |a_2(\bv,\bv)|.
\end{equation}
The imaginary part of $a_2(\textbf{v},\textbf{v})$ is given by
$$
Im(a_2(\textbf{v},\textbf{v})) = \sum_{T\in {\cal T}_h} (\omega\mu
\bv_0,\bv_0)_T- (\omega\varepsilon(\sigma^2 + \omega^2
\varepsilon^2)^{-1}\nabla_{w}\times \textbf{v},\nabla_{w} \times
\textbf{v})_T,
$$
which leads to
\begin{equation}\label{EQ:10.11.2016:011}
\left|\sum_{T\in {\cal T}_h} (\omega\mu \bv_0,\bv_0)_T-
(\omega\varepsilon(\sigma^2 + \omega^2
\varepsilon^2)^{-1}\nabla_{w}\times \textbf{v},\nabla_{w} \times
\textbf{v})_T\right| \leq |a_2(\bv,\bv)|.
\end{equation}
Combining (\ref{EQ:10.11.2016:010}) with (\ref{EQ:10.11.2016:011})
gives rise to the coercivity estimate (\ref{EQ:atwo-coercivity}).
This completes the proof of the lemma.
\end{proof}

\medskip

Next, we establish an {\em inf-sup} condition for the bilinear form
$b_1(\cdot,\cdot)$ used in the WG algorithm \ref{algo1}. To this
end, for any $q\in W^1_h$, set $\bv_q=\{-h^2\nabla q;\ h
\bv_{q,b}\}\in \bV_{h,0}^1$, where
\begin{equation}\label{November-30:500}
\bv_{q,b}=\left\{
\begin{array}{ll}
\jump{q}\ \bn_e,&\qquad \mbox{ on } \ e\in\E_h^0,\\
 q \ \bn,&\qquad \mbox{ on } \ e\in \E_h\cap \Gamma.
\end{array}
\right.
\end{equation}
Now for any $\bv=\{\bv_0; \bv_b\}\in \bV_{h,0}^1$, from the
definition (\ref{2.2}) of weak divergence, we have
\begin{equation}\label{November.30.001-new}
\begin{split}
b_1(\bv,q) & = \sum_{T\in {\cal T}_h} (\nabla_w \cdot ( \varepsilon\bv),q)_T\\
& = \sum_{T\in {\cal T}_h} -( \varepsilon\bv_0,\nabla q)_T+\langle
\bv_b
\cdot \bn, q\rangle_{\partial T}\\
& = - \sum_{T\in {\cal T}_h} ( \varepsilon\bv_0,\nabla q)_T +
\sum_{e\in\E_h} \langle \bv_b\cdot\bn_e, \jump{q}\rangle_e.
\end{split}
\end{equation}

\begin{lemma} \emph{(inf-sup condition for WG algorithm \ref{algo1})} \label{lvq2}
For any $q\in W_h^1$, there exists a finite element function
$\textbf{v}_q\in \bV_{h,0}^1$ such that
\begin{eqnarray}\label{November:29:801}
b_1(\bv_q,q) &=& h^2\sum_{T\in {\cal T}_h} ( \varepsilon\nabla
q,\nabla q)_T +
h\sum_{e\in \E_h} \|\jump{q}\|_e^2, \\
\label{6.5} \3bar\bv_q\3bar_{\bV_{h, 0}^1}  &\leqC &\|q\|_{W_h^1}.
\end{eqnarray}
\end{lemma}

\begin{proof} For any $q\in W_h^1$, we define $\bv_{q,b}$ by (\ref{November-30:500}) and set
$\textbf{v}_q = \{-h^2\nabla q;\ h \bv_{q,b}\}\in \bV_{h,0}^1$. By
letting $\bv=\bv_q$ in (\ref{November.30.001-new}) we obtain
\begin{equation*}
b_1(\bv_q,q)=h^2\sum_{T\in {\cal T}_h} ( \varepsilon\nabla q,\nabla
q)_T + h\sum_{e\in \E_h} \|\jump{q}\|_e^2,
\end{equation*}
which verifies the identity (\ref{November:29:801}).

To derive (\ref{6.5}), we consider the following decomposition
$$
\bv_q = \bv_q^{(1)} + \bv_q^{(2)},
$$
where $\bv_q^{(1)} = -\{h^2\nabla q;\ 0\}$ and $\bv_q^{(2)} = \{0;\
h\bv_{q,b}\}$. It suffices to establish (\ref{6.5}) for
$\bv_q^{(1)}$ and $\bv_q^{(2)}$ respectively. Using
(\ref{Eq:bar-norm}), we have
\begin{equation}\label{vq2}
\begin{split}
\3bar \textbf{v}_q^{(1)} \3bar_{\bV_{h, 0}^1} ^2 = & \sum_{T\in
{\cal T}_h}
\|\nabla_w \times \textbf{v}_q^{(1)}\|^2_T+ \|   h^2 \nabla q\|^2_T \\
& + h_T^{-1}\|h^2  \varepsilon\nabla q\cdot\bn \|^2_{\partial T} +
h_T^{-1}\|h^2\nabla q \times\bn\|^2_\pT.
\end{split}
\end{equation}
It follows from (\ref{2.5}) of the discrete weak curl that
$$
(\nabla_w \times \textbf{v}_q^{(1)}, \varphi)_T = -h^2 (\nabla q,
\curl\varphi)_T,\qquad \forall \ \varphi\in [P_{k-1}(T)]^d.
$$
Using the inverse inequality (\ref{Div-Curl:inverse}) we obtain
$$
\|\nabla_w \times \textbf{v}_q^{(1)}\|_T \leqC h \|\nabla q\|_T.
$$
Substituting the above inequality into (\ref{vq2}) and then using
the trace inequality (\ref{Aa-trace}) gives rise to
$$
\3bar \textbf{v}_q^{(1)} \3bar_{\bV_{h, 0}^1} ^2 \leqC h^2\|\nabla
q\|_T^2,
$$
which verifies the estimate (\ref{6.5}) for $\textbf{v}_q^{(1)}$.

For $\textbf{v}_q^{(2)}$, we again use  (\ref{Eq:bar-norm}) to
obtain
\begin{equation}\label{vq2-new}
\3bar \textbf{v}_q^{(2)} \3bar_{\bV_{h, 0}^1} ^2= \sum_{T\in {\cal
T}_h} \|\nabla_w \times \textbf{v}_q^{(2)}\|^2_T + h_T^{-1}\|h
\bv_{q,b}\cdot\bn \|^2_{\partial T} + h_T^{-1}\|h \bv_{q,b}\times\bn
\|^2_{\pT}.
\end{equation}
Since $\bv_{q,b}$ is parallel to $\bn$, then $\bv_{q,b}\times\bn=0$
on $\pT$. In addition,  (\ref{2.5}) of the discrete weak curl
implies $\nabla_w \times \textbf{v}_q^{(2)}=0$, since
$$
(\nabla_w \times \textbf{v}_q^{(2)}, \varphi)_T = (0,
\curl\varphi)_T- h \langle \bv_{q,b}\times\bn,
\varphi\rangle_{\pT}=0 ,\quad \forall \ \varphi\in [P_{k-1}(T)]^d.
$$
Thus, it follows from (\ref{vq2-new}) and (\ref{November-30:500})
that
$$
\3bar \textbf{v}_q^{(2)} \3bar_{\bV_{h, 0}^1} ^2 \leqC h \sum_{e\in
\E_h} \|\jump{q}\|_e^2,
$$
which verifies the estimate (\ref{6.5}) for $\textbf{v}_q^{(2)}$.
This completes the proof of the lemma.
\end{proof}

For the bilinear form $b_2(\cdot,\cdot)$, we may follow the same
spirit of Lemma \ref{lvq2} to derive an {\em inf-sup} condition. For
completeness, we present all the necessary details as follows. For
any $q\in W^2_h$, define a finite element function
 $\bv_q=\{-h^2\nabla q;\ h
\bv_{q,b}\}\in \bV_{h,0}^2$, where
\begin{equation}\label{November-30:500-2}
\bv_{q,b}=\left\{
\begin{array}{ll}
\jump{q}\ \bn_e,&\qquad \mbox{ on } \ e\in\E_h^0,\\
0,&\qquad \mbox{ on } \ e\in \E_h\cap \Gamma.
\end{array}
\right.
\end{equation}

Note that for any $\bv=\{\bv_0; \bv_b\}\in \bV_{h,0}^2$, from
(\ref{2.2}) of weak divergence, we have
\begin{equation}\label{November.30.001-new-2}
\begin{split}
b_2(\bv,q) & = \sum_{T\in {\cal T}_h} (\nabla_w \cdot ( \mu\bv),q)_T\\
& = \sum_{T\in {\cal T}_h} -( \mu\bv_0,\nabla q)_T+\langle \bv_b
\cdot \bn, q\rangle_{\partial T}\\
& = - \sum_{T\in {\cal T}_h} ( \mu\bv_0,\nabla q)_T +
\sum_{e\in\E_h^0} \langle \bv_b\cdot\bn_e, \jump{q}\rangle_e.
\end{split}
\end{equation}

\begin{lemma} \emph{(inf-sup condition for WG algorithm \ref{algo2})} \label{lvq2-2}
For any $q\in W_h^2$, there exists a finite element function
$\textbf{v}_q\in \bV_{h,0}^2$ such that
\begin{eqnarray}\label{November:29:801-2}
b_2(\bv_q,q) &=& h^2\sum_{T\in {\cal T}_h} ( \mu\nabla q,\nabla q)_T
+
h\sum_{e\in \E_h^0} \|\jump{q}\|_e^2 , \\
\label{6.5-2} \3bar\bv_q\3bar_{\bV_{h, 0}^2}  &\leqC &\|q\|_{W_h^2}.
\end{eqnarray}
\end{lemma}

\begin{proof} For any $q\in W_h^2$, we define $\bv_{q,b}$ by
(\ref{November-30:500-2}) and set $\textbf{v}_q = \{-h^2\nabla q;\ h
\bv_{q,b}\}$. It is easy to see that $\bv_q\in\bV_{h,0}^2$. By
letting $\bv=\bv_q$ in (\ref{November.30.001-new-2}) we arrive at
\begin{equation*}
b_2(\bv_q,q)=h^2\sum_{T\in {\cal T}_h} ( \mu\nabla q,\nabla q)_T +
h\sum_{e\in \E_h^0} \|\jump{q}\|_e^2,
\end{equation*}
which verifies the identity (\ref{November:29:801-2}).

To derive (\ref{6.5-2}), we consider the following decomposition
$$
\bv_q = \bv_q^{(1)} + \bv_q^{(2)},
$$
where $\bv_q^{(1)} = -\{h^2\nabla q;\ 0\}$ and $\bv_q^{(2)} = \{0;\
h\bv_{q,b}\}$. It suffices to establish (\ref{6.5-2}) for
$\bv_q^{(1)}$ and $\bv_q^{(2)}$ respectively. From
(\ref{Eq:bar-norm-2}), we have
\begin{equation}\label{vq2-2}
\begin{split}
\3bar \textbf{v}_q^{(1)} \3bar_{\bV_{h, 0}^2} ^2 = & \sum_{T\in
{\cal T}_h}
\|\nabla_w \times \textbf{v}_q^{(1)}\|^2_T+ \|   h^2 \nabla q\|^2_T \\
& + h_T^{-1}\|h^2  \mu\nabla q\cdot\bn \|^2_{\partial T} +
h_T^{-1}\|h^2\nabla q \times\bn\|^2_\pT.
\end{split}
\end{equation}
The definition (\ref{2.5}) for the discrete weak curl implies
$$
(\nabla_w \times \textbf{v}_q^{(1)}, \varphi)_T = -h^2 (\nabla q,
\curl\varphi)_T,\qquad \forall \ \varphi\in [P_{k-1}(T)]^d.
$$
It follows from the inverse inequality (\ref{Div-Curl:inverse}) that
$$
\|\nabla_w \times \textbf{v}_q^{(1)}\|_T \leqC h \|\nabla q\|_T.
$$
Substituting the above into (\ref{vq2-2}) and then using the trace
inequality (\ref{Aa-trace}) yields
$$
\3bar \textbf{v}_q^{(1)} \3bar_{\bV_{h, 0}^2} ^2 \leqC h^2\|\nabla
q\|_T^2,
$$
which verifies the estimate (\ref{6.5-2}) for $\textbf{v}_q^{(1)}$.

For $\textbf{v}_q^{(2)}$, we again use (\ref{Eq:bar-norm-2}) to
obtain
\begin{equation}\label{vq2-new-2}
\3bar \textbf{v}_q^{(2)} \3bar_{\bV_{h, 0}^2} ^2= \sum_{T\in {\cal
T}_h} \|\nabla_w \times \textbf{v}_q^{(2)}\|^2_T + h_T^{-1}\|h
\bv_{q,b}\cdot\bn \|^2_{\partial T} + h_T^{-1}\|h \bv_{q,b}\times\bn
\|^2_{\pT}.
\end{equation}
Since $\bv_{q,b}$ is parallel to $\bn$, then $\bv_{q,b}\times\bn=0$
on $\pT$. In addition, the definition (\ref{2.5}) for the discrete
weak curl implies $\nabla_w \times \textbf{v}_q^{(2)}=0$, since
$$
(\nabla_w \times \textbf{v}_q^{(2)}, \varphi)_T = (0,
\curl\varphi)_T- h \langle \bv_{q,b}\times\bn,
\varphi\rangle_{\pT}=0 ,\quad \forall \ \varphi\in [P_{k-1}(T)]^d.
$$
Thus, it follows from (\ref{vq2-new-2}) and
(\ref{November-30:500-2}) that
$$
\3bar \textbf{v}_q^{(2)} \3bar_{\bV_{h, 0}^2}  ^2 \leqC h
\sum_{e\in \E_h^0} \|\jump{q}\|_e^2,
$$
which verifies the estimate (\ref{6.5-2}) for $\textbf{v}_q^{(2)}$.
This completes the proof of the lemma.
\end{proof}

Using the general result of Babu\u{s}ka \cite{babuska} and Brezzi
\cite{brezzi} we obtain the following result on the solution
existence and uniqueness for our WG finite element algorithms.

\begin{theorem} The weak Galerkin algorithm \ref{algo1} or the system of equations (\ref{3.3})-(\ref{3.4})
has a unique solution. The same conclusion can be drawn for the weak
Galerkin algorithm \ref{algo2} or the system of equations
(\ref{3.5})-(\ref{3.6}).
\end{theorem}

\section{Error Equations}\label{Section:sec7}
In this section we shall establish two error equations for the weak
Galerkin algorithms \ref{algo1} and \ref{algo2}. These error
equations will be used for deriving error estimates for the
resulting numerical schemes.

Let $Q_0$ be the $L^2$ projection onto $[P_k(T)]^d, \ T\in {\cal
T}_h$, and $Q_b$ be the $L^2$ projection onto $[P_{k}(e)]^d, \ e\in
\partial T\cap \E_h$. Denote by $Q_h$ the $L^2$ projection onto the
weak finite element space $\bV_h$ such that on each element $T\in
{\cal T}_h$,
\begin{equation}\label{EQ:12-31:001}
(Q_h \textbf{u})|_T= \{Q_{0}\textbf{u}, \mathds{Q}_{b} \textbf{u}\},
\end{equation}
where
\begin{equation}\label{EQ:12-31:002}
\mathds{Q}_{b} \textbf{u} = Q_b ( \varepsilon\bu\cdot\bn)\bn + Q_b
(\bn\times(\bu\times\bn)).
\end{equation}
Observe that $\bn\times(\bu\times\bn)=\bu-(\bu\cdot\bn)\bn$ is the
tangential component of the vector $\bu$ on the boundary of the
element. In the case of $ \varepsilon=I$,
$(\varepsilon\bu\cdot\bn)\bn$ is clearly the normal component of the
vector $\bu$. But for general $\varepsilon$,
$(\varepsilon\bu\cdot\bn)\bn +\bn\times(\bu\times\bn)$ is not a
decomposition of the vector $\bu$ on $\partial T$.

Denote by ${\cal Q}_h$ and $\textbf{Q}_h$ the $L^2$ projections onto
$P_{k-1}(T)$ and $[P_{k-1}(T)]^{d}$, respectively.

\begin{lemma}\cite{mwyz, cwang1, wy1302} \label{lemma4.2} The $L^2$
projection operators $Q_h$, $\textbf{Q}_h$, and ${\cal Q}_h$ satisfy
the following commutative identities:
\begin{align}\label{4.4}
\nabla_w \cdot( \varepsilon Q_h \textbf{v})=&{\cal Q}_h( \nabla \cdot
( \varepsilon\textbf{v})), \qquad \textbf{v}\in  H({\rm div}_ \varepsilon;\Omega),\\
\nabla_w \cdot( \mu Q_h \textbf{v})=&{\cal Q}_h( \nabla \cdot
( \mu\textbf{v})), \qquad \textbf{v}\in  H({\rm div}_ \mu;\Omega),\label{4.4-2}\\
\nabla_w \times (Q_h \textbf{v})=& \textbf{Q}_h (\nabla \times
\textbf{v}), \qquad \textbf{v}\in  H({\rm curl};\Omega).\label{4.5}
\end{align}
\end{lemma}

Let $(\bu_h;p_h)=(\{\bu_0, \bu_b\};p_h)$ be the WG finite element
solution arising from either the weak Galerkin Algorithm
(\ref{3.3})-(\ref{3.4}) or (\ref{3.5})-(\ref{3.6}), and $(\bu;p)$ be
the solution of the continuous model problem (\ref{maxwell1}) or
(\ref{maxwell2}). Here the variable $\bu$ represents either the
electric field intensity $\bE$ or the magnetic field intensity
$\bH$. The corresponding error functions are given as follows
\begin{align}\label{5.1}
\textbf{e}_h&=\{\textbf{e}_0,\textbf{e}_b\}=\{Q_0\textbf{u}-\textbf{u}_0,
\mathds{Q}_{b}\textbf{u}-\textbf{u}_b\},\\
 \epsilon_h&={\cal Q}_h p-p_h.\label{ph}
\end{align}

\begin{lemma}\label{lemma5.1}
 Assume that $(\bw\ ; \rho)\in  (H_0({\rm curl};\Omega)\cap H(\rm div_\varepsilon; \Omega) )\times
L^2(\Omega)$ is sufficiently smooth on each element $T\in \T_h$ and
satisfies
\begin{eqnarray}
 \nabla\times(\mu^{-1}\nabla \times \bw)+(i \omega \sigma -\omega^2 \varepsilon) \bw +
 \varepsilon\nabla \rho &= &\eta,\qquad
 \mbox{in}\ \Omega, \label{5.2}\\
 \rho&=& 0,\qquad \mbox{on} \ \Gamma.
\label{November.28.100}
\end{eqnarray}
Then, the following identity holds true:
\begin{equation}\label{Div-Curl:Feb9:300-new}
\begin{split}
&\sum_{T\in {\cal T}_h}(\mu^{-1}\nabla_w \times (Q_h\bw),\nabla_w \times\textbf{v}
)_T+ ((i \omega \sigma -\omega^2 \varepsilon )Q_0 \bw,\bv_0)_T\\
&-(\nabla_w\cdot ( \varepsilon\textbf{v}),{\cal
Q}_h\rho)_T=(\eta,\textbf{v}_0)+l_\bw
(\textbf{v})-\theta_\rho(\textbf{v}),
\end{split}
\end{equation}
for all $\textbf{v}\in \bV_{h,0}^1$. Here $l_\bw(\textbf{v})$ and
$\theta_\rho(\textbf{v})$ are two  functionals in the linear space
$ \bV_{h,0}^1$ given by
\begin{align}\label{l}
l_\bw(\textbf{v})&= \sum_{T\in{\cal T}_h}\langle
(\textbf{Q}_h-I)(\mu^{-1}\nabla \times\bw),
(\textbf{v}_0-\textbf{v}_b)\times \textbf{n}\rangle_{\partial T},\\
\theta_\rho(\textbf{v})&=\sum_{T\in{\cal T}_h}\langle \rho- {\cal
Q}_h\rho,
 ( \varepsilon\textbf{v}_0-\textbf{v}_b)\cdot\bn\rangle_{\partial T}.\label{theta}
\end{align}
\end{lemma}

\begin{proof} Recall that  $\mu$, $\sigma$, $\omega$ and $\varepsilon$ are assumed to be
piecewise constants on the domain $\Omega$ with respect to the given
finite element partition. Thus, from (\ref{2.5new}) with $\varphi=
\mu^{-1}\nabla_w \times (Q_h\textbf{w})$ we have
\begin{equation}\label{Eq:Feb8:500}\nonumber
\begin{split}
 (\nabla_w \times & \textbf{v},  \mu^{-1}\nabla_w \times
(Q_h\textbf{w}))_T = \\
 & \ (\nabla \times \textbf{v}_0,  \mu^{-1}\nabla_w \times (Q_0\textbf{w}))
_T -\langle (\textbf{v}_b-\textbf{v}_0)\times\bn,
 \mu^{-1}\nabla_w \times (Q_h\textbf{w})\rangle_{\partial T}.
\end{split}
\end{equation}
Using (\ref{4.5}), the above equation can be rewritten as
\begin{equation}\label{5.4}\nonumber
\begin{split}
( \mu^{-1}\nabla_w \times & (Q_h\textbf{w}),\nabla_w \times \textbf{v}
)_T = \\
  & ( \mu^{-1}\nabla\times \bw, \nabla \times \bv_0)_T+\langle
 \textbf{Q}_h (\mu^{-1}\nabla\times \bw), (\bv_0-\bv_b)\times \bn
\rangle_{\partial T}.
\end{split}
\end{equation}
Applying the integration by parts to the first term on the
right-hand side yields
\begin{equation}\label{5.4.001}
\begin{split}
&( \mu^{-1}\nabla_w \times  (Q_h\textbf{w}),\nabla_w \times\textbf{v}
)_T+((i \omega \sigma -\omega^2 \varepsilon )Q_0 \bw,\bv_0)_T \\
 = & (\curl(\mu^{-1}\curl\bw), \bv_0)_T  - \langle \mu^{-1}\curl \bw, \bv_0\times\bn\rangle_\pT\\
 & +\langle \textbf{Q}_h (\mu^{-1}\nabla\times \bw), (\bv_0-\bv_b)\times \bn
\rangle_{\partial T}
+((i \omega \sigma -\omega^2 \varepsilon )  \bw,\bv_0)_T\\
= & (\curl(\mu^{-1}\curl\bw), \bv_0)_T - \langle \mu^{-1}\curl \bw,
\bv_b\times\bn\rangle_\pT\\
& + \langle (\textbf{Q}_h-I) (\mu^{-1}\nabla\times \bw),
(\bv_0-\bv_b)\times \bn \rangle_{\partial T}+((i \omega \sigma -\omega^2 \varepsilon )  \bw,\bv_0)_T.
\end{split}
\end{equation}

Using (\ref{2.2new}) with $\varphi={\cal Q}_h \rho$ and the usual
integration by parts, we obtain
\begin{equation}\label{Eq:Feb8:501}
\begin{split}
&\ (\nabla_w\cdot( \varepsilon\bv),{\cal Q}_h \rho)_T\\
  = &\ (\nabla \cdot( \varepsilon\bv_0), {\cal Q}_h \rho)_T+\langle
(\bv_b- \varepsilon\bv_0)\cdot\bn,{\cal Q}_h
\rho\rangle_{\partial T}\\
 =&\ (\nabla \cdot ( \varepsilon\bv_0), \rho)_T+\langle
(\bv_b- \varepsilon\bv_0)\cdot\bn,{\cal
Q}_h\rho\rangle_{\partial T}\\
  =&\ -( \varepsilon\bv_0, \nabla \rho)_T+\langle  \varepsilon\bv_0\cdot\bn,
\rho\rangle_{\partial T} +\langle (\bv_b- \varepsilon\bv_0)\cdot\bn,{\cal
Q}_h
\rho\rangle_{\partial T}\\
 =&\ -(\bv_0,  \varepsilon\nabla \rho)_T+\langle
(\bv_b- \varepsilon\bv_0)\cdot\bn,{\cal Q}_h \rho-\rho\rangle_{\partial T}
+\langle \bv_b\cdot\bn, \rho\rangle_{\partial T}.
\end{split}
\end{equation}

Summing (\ref{5.4.001}) over all the elements $T\in\T_h$ yields
\begin{equation}\label{5.4.002}
\begin{split}
&\sum_{T\in {\cal T}_h}( \mu^{-1}\nabla_w \times  (Q_h\textbf{w}),\nabla_w \times \textbf{v}
)_T + ((i \omega \sigma -\omega^2 \varepsilon )Q_0 \bw,\bv_0)_T\\
=& \sum_{T\in\T_h} (\curl(\mu^{-1}\curl\bw), \bv_0)_T +  \langle
 (\textbf{Q}_h-I) (\mu^{-1}\nabla\times \bw), (\bv_0-\bv_b)\times \bn
\rangle_{\partial T}\\
&+  ((i \omega \sigma -\omega^2 \varepsilon ) \bw,\bv_0)_T,
\end{split}
\end{equation}
where we have used two properties: (1) the cancelation property for
the boundary integrals on interior edges/faces, and (2) the fact
that $\bv_b\times\bn=0$ on $\Gamma$. Similarly, summing
(\ref{Eq:Feb8:501}) over all the elements $T\in\T_h$ leads to
\begin{equation}\label{5.4.003}
\begin{split}
\sum_{T\in {\cal T}_h} (\nabla_w\cdot ( \varepsilon\textbf{v}),{\cal Q}_h \rho)_T = &
-(\bv_0, \varepsilon\nabla \rho) + \sum_{T\in {\cal T}_h} \langle
( \varepsilon\bv_0-\bv_b)\cdot\bn,
\rho-{\cal Q}_h \rho \rangle_{\partial T}\\
& + \sum_{e\in \E_h\cap \Gamma} \langle \bv_b\cdot\bn,
\rho\rangle_e.
\end{split}
\end{equation}
The third term on the right-hand side of (\ref{5.4.003}) vanishes if
$\rho$ satisfies the boundary condition
(\ref{November.28.100}). Thus, the equation
(\ref{Div-Curl:Feb9:300-new}) holds true from
(\ref{5.4.002}) and (\ref{5.4.003}). This completes the proof of the
lemma.
\end{proof}

\begin{lemma}\label{lemma5.1-2}
Assume that $(\bw; \rho)\in (H_0({\rm div_{\mu}};\Omega) \cap H(\rm
curl; \Omega))\times L_0^2(\Omega)$ is sufficiently smooth on each
element $T\in \T_h$ and satisfies
\begin{eqnarray}
\nabla\times((i\omega \varepsilon+\sigma)^{-1} \nabla \times \bw) +
i \omega \mu \bw+ \mu \nabla \rho &=& \eta,\qquad
 \mbox{in}\ \Omega, \label{5.2-2}\\
(i\omega \varepsilon+\sigma)^{-1} \nabla \times \bw \times \bn &=&
0,\qquad \mbox{on}\ \Gamma. \label{5.2-21}
\end{eqnarray}
Then, we have the following identity:
\begin{equation}\label{Div-Curl:Feb9:300-new-2}
\begin{split}
&\sum_{T\in {\cal T}_h}((i\omega \varepsilon+\sigma)^{-1}\nabla_w \times (Q_h\bw),\nabla_w \times\textbf{v}
)_T+ ( i \omega \mu Q_0 \bw,\bv_0)_T\\
&-(\nabla_w\cdot ( \mu\textbf{v}),{\cal
Q}_h\rho)_T=(\eta,\textbf{v}_0)+l'_\bw
(\textbf{v})-\theta'_\rho(\textbf{v}),
\end{split}
\end{equation}
for all $\textbf{v}\in \bV_{h,0}^2$. Here $l'_\bw(\textbf{v})$ and
$\theta'_\rho(\textbf{v})$ are two functionals in the linear space
$ \bV_{h,0}^2$ given by
\begin{align}\label{l-2}
l'_\bw(\textbf{v})&= \sum_{T\in{\cal T}_h}\langle
(\textbf{Q}_h-I) (i\omega \varepsilon+\sigma)^{-1}\nabla \times\bw ,
(\textbf{v}_0-\textbf{v}_b)\times \textbf{n}\rangle_{\partial T},\\
\theta'_\rho(\textbf{v})&=\sum_{T\in{\cal T}_h}\langle \rho- {\cal
Q}_h\rho,
 ( \mu\textbf{v}_0-\textbf{v}_b)\cdot\bn\rangle_{\partial T}.\label{theta-2}
\end{align}
\end{lemma}

\begin{proof} Since  $\mu$, $\sigma$, $\omega$ and $\varepsilon$ are piecewise constants on
the domain $\Omega$ with respect to the given finite element
partitions, then from (\ref{2.5new}) with $\varphi= (i\omega
\varepsilon+\sigma)^{-1}\nabla_w \times (Q_h\textbf{w})$ we obtain
\begin{equation}\label{Eq:Feb8:500-2}\nonumber
\begin{split}
  (\nabla_w \times  \textbf{v}, (i\omega \varepsilon+&\sigma)^{-1}\nabla_w \times
(Q_h\textbf{w}))_T =\ (\nabla \times \textbf{v}_0,  (i\omega \varepsilon+\sigma)^{-1}\nabla_w \times (Q_h\textbf{w}))
_T \\
 &  -\langle (\textbf{v}_b-\textbf{v}_0)\times\bn,
 (i\omega \varepsilon+\sigma)^{-1}\nabla_w \times (Q_h\textbf{w})\rangle_{\partial T},
\end{split}
\end{equation}
which, combined with  (\ref{4.5}),  gives rise to
\begin{equation}\label{5.4-2}\nonumber
\begin{split}
 &((i\omega \varepsilon+\sigma)^{-1}\nabla_w \times (Q_h\textbf{w}),\nabla_w \times \textbf{v}
)_T = \\
  & ((i\omega \varepsilon+\sigma)^{-1}\nabla\times \bw, \nabla \times \bv_0)_T+\langle
 \textbf{Q}_h ((i\omega \varepsilon+\sigma)^{-1}\nabla\times \bw), (\bv_0-\bv_b)\times \bn
\rangle_{\partial T}.
\end{split}
\end{equation}
Now applying the integration by parts to the first term on the
right-hand side of the above identity yields
\begin{equation}\label{5.4.001-2}
\begin{split}
&((i\omega \varepsilon+\sigma)^{-1}\nabla_w \times  (Q_h\textbf{w}),\nabla_w \times\textbf{v}
)_T +(i \omega \mu Q_0 \bw, \bv_0)_T\\
 = & (\curl((i\omega \varepsilon+\sigma)^{-1}\curl\bw), \bv_0)_T  - \langle (i\omega \varepsilon+\sigma)^{-1}\curl \bw, \bv_0\times\bn\rangle_\pT\\
 & +\langle \textbf{Q}_h ((i\omega \varepsilon+\sigma)^{-1}\nabla\times \bw), (\bv_0-\bv_b)\times \bn
\rangle_{\partial T}+(i \omega \mu  \bw, \bv_0)_T\\
= & (\curl((i\omega \varepsilon+\sigma)^{-1}\curl\bw), \bv_0)_T - \langle (i\omega \varepsilon+\sigma)^{-1}\curl \bw,
\bv_b\times\bn\rangle_\pT\\
& + \langle (\textbf{Q}_h-I) ((i\omega \varepsilon+\sigma)^{-1}\nabla\times \bw),
(\bv_0-\bv_b)\times \bn \rangle_{\partial T}+(i \omega \mu \bw, \bv_0)_T.
\end{split}
\end{equation}

Using (\ref{2.2new}) with $\varphi={\cal Q}_h \rho$ and the usual
integration by parts, we obtain
\begin{equation}\label{Eq:Feb8:501-2}
\begin{split}
&\ (\nabla_w\cdot( \mu\bv),{\cal Q}_h \rho)_T\\
  = &\ (\nabla \cdot( \mu\bv_0), {\cal Q}_h \rho)_T+\langle
(\bv_b- \mu\bv_0)\cdot\bn,{\cal Q}_h
\rho\rangle_{\partial T}\\
 =&\ (\nabla \cdot ( \mu\bv_0), \rho)_T+\langle
(\bv_b- \mu\bv_0)\cdot\bn,{\cal
Q}_h\rho\rangle_{\partial T}\\
  =&\ -( \mu\bv_0, \nabla \rho)_T+\langle  \mu\bv_0\cdot\bn,
\rho\rangle_{\partial T} +\langle (\bv_b- \mu\bv_0)\cdot\bn,{\cal
Q}_h
\rho\rangle_{\partial T}\\
 =&\ -(\bv_0,  \mu\nabla \rho)_T+\langle
(\bv_b- \mu\bv_0)\cdot\bn,{\cal Q}_h \rho-\rho\rangle_{\partial T}
+\langle \bv_b\cdot\bn, \rho\rangle_{\partial T}.
\end{split}
\end{equation}

Summing (\ref{5.4.001-2}) over all the elements $T\in\T_h$ yields
\begin{equation}\label{5.4.002-2}
\begin{split}
&\sum_{T\in {\cal T}_h}((i\omega \varepsilon+\sigma)^{-1}\nabla_w \times  (Q_h\textbf{w}),\nabla_w \times \textbf{v}
)_T+  (i \omega \mu Q_0 \bw, \bv_0)_T \\
=& \sum_{T\in\T_h} (\curl((i\omega \varepsilon+\sigma)^{-1}\curl\bw), \bv_0)_T  \\
  & + \langle
 (\textbf{Q}_h-I) ((i\omega \varepsilon+\sigma)^{-1}\nabla\times \bw), (\bv_0-\bv_b)\times \bn
\rangle_{\partial T}+(i \omega \mu \bw,  \bv_0)_T,
\end{split}
\end{equation}
where we have used two properties: the first is the cancelation
property for the boundary integrals on interior edges/faces, and the
second is the boundary condition (\ref{5.2-21}). Similarly, summing
(\ref{Eq:Feb8:501-2}) over all the elements $T\in\T_h$, we obtain
\begin{equation}\label{5.4.003-2}
\begin{split}
\sum_{T\in {\cal T}_h}(\nabla_w\cdot ( \mu\textbf{v}),{\cal Q}_h \rho)_T = &
-(\bv_0, \mu\nabla \rho) + \sum_{T\in {\cal T}_h} \langle
( \mu\bv_0-\bv_b)\cdot\bn,
\rho-{\cal Q}_h \rho \rangle_{\partial T}\\
& + \sum_{e\in \E_h\cap \Gamma} \langle \bv_b\cdot\bn,
\rho\rangle_e.
\end{split}
\end{equation}
The third term on the right-hand side of (\ref{5.4.003-2}) vanishes
for $\bv\in \bV_{h,0}^2$. Thus, the equation
(\ref{Div-Curl:Feb9:300-new-2}) holds true from (\ref{5.4.002-2})
and (\ref{5.4.003-2}). This completes the proof of the lemma.
\end{proof}

\begin{theorem} \label{Thm:div-curl:theorem-error-eqns}
Let $(\bu; p)$ be the solution of the problem (\ref{maxwell1}) for
the electric field and $(\bu_h; p_h)$ be its numerical solution
arising from the WG finite element scheme (\ref{3.3})-(\ref{3.4}).
Define the error functions $\textbf{e}_h$ and $\epsilon_h$ by
(\ref{5.1})-(\ref{ph}). Then, $\be_h\in\bV_{h,0}^1$ and the
following error equations hold true:
\begin{align}\label{EQ:div-curl:error-eq-01}
a_1(\textbf{e}_h,\textbf{v})-b_1(\textbf{v},\epsilon_h)&=
\varphi_{\textbf{u},p}(\textbf{v}), \qquad \forall \textbf{v}\in \bV_{h,0}^1,\\
b_1(\textbf{e}_h,q)&=0,\qquad \qquad \ \forall q\in
W_h^1,\label{EQ:div-curl:error-eq-02}
\end{align}
where
\begin{equation}\label{EQ:div-cul:varphi-up}
\varphi_{\textbf{u},p}(\textbf{v})=l_\textbf{u}(\textbf{v})-
\theta_p(\textbf{v})+ s_1(Q_h\textbf{u},\textbf{v}).
\end{equation}
\end{theorem}

\begin{proof}
 Let $(\bu; p)$ be the
solution of the model problem (\ref{maxwell1}). It is not hard to see that the
following holds true:
\begin{equation}\nonumber\label{5.2.Nov-28.200}
\begin{split}
  \nabla\times( \mu^{-1} \nabla \times \bu) + (i \omega \sigma -\omega^2 \varepsilon )  \bu + \varepsilon
  \nabla p &  = -i\omega \textbf{j}_e,\quad
 \mbox{in}\ \Omega,\\
  p & =0,\qquad\quad \mbox{on}\ \Gamma.
\end{split}
\end{equation}
It follows from Lemma \ref{lemma5.1} that
\begin{equation*}
\begin{split}
&\sum_{T\in {\cal T}_h}(\mu^{-1}\nabla_w\times(Q_h\textbf{u}),\nabla_w\times
\textbf{v})_T+ ((i \omega \sigma -\omega^2 \varepsilon )  Q_0\bu,\bv_0)_T-(\nabla_w \cdot ( \varepsilon\textbf{v}),{\cal
Q}_hp)_T\\
=&(-i\omega \textbf{j}_e,\textbf{v}_0)
+l_\textbf{u}(\textbf{v})-\theta_p(\textbf{v}),
\end{split}
\end{equation*}
for all $\bv\in  \bV_{h,0}^1$, which gives
\begin{equation}\label{4.13}
a_1(Q_h\textbf{u},\textbf{v})-b_1(\textbf{v},{\cal
Q}_hp)=(-i\omega \textbf{j}_e,\textbf{v}_0)+l_\textbf{u}(\textbf{v})-
\theta_p(\textbf{v})+s_1(Q_h\textbf{u},\textbf{v}).
\end{equation}
Subtracting (\ref{3.3}) from (\ref{4.13}) gives rise to the first
error equation (\ref{EQ:div-curl:error-eq-01}).

Next, from the second equation in (\ref{maxwell1}) and the
commutative relation (\ref{4.4}), we have for any $q\in W_h^1$,
 \begin{equation}\label{4.14}
(\rho,q)=\sum_{T\in {\cal T}_h}(\nabla\cdot ( \varepsilon\textbf{u}),q)_T=\sum_{T\in {\cal T}_h}({\cal Q}_h(\nabla\cdot
( \varepsilon\textbf{u})),q)_T=\sum_{T\in {\cal T}_h}(\nabla_w\cdot( \varepsilon Q_h\textbf{u}),q)_T.
\end{equation}
The difference of (\ref{4.14}) and (\ref{3.4}) yields the
second error equation (\ref{EQ:div-curl:error-eq-02}). This
completes the proof.
\end{proof}

\begin{theorem} \label{Thm:div-curl:theorem-error-eqns-2}
Let $(\bu; p)$ be the solution of the problem (\ref{maxwell2}) for
the magnetic field and $(\bu_h; p_h)$ be its numerical solution
arising from the WG finite element scheme (\ref{3.5})-(\ref{3.6}).
Denote the error functions $\textbf{e}_h$ and $\epsilon_h$ by
(\ref{5.1})-(\ref{ph}). Then, $\be_h\in\bV_{h,0}^2$ and the
following error equations hold true:
\begin{align}\label{EQ:div-curl:error-eq-01-2}
a_2(\textbf{e}_h,\textbf{v})-b_2(\textbf{v},\epsilon_h)&=
\varphi'_{\textbf{u},p}(\textbf{v}), \qquad \forall \textbf{v}\in \bV_{h,0}^2,\\
b_2(\textbf{e}_h,q)&=0,\qquad \qquad \ \forall q\in
W_h^2,\label{EQ:div-curl:error-eq-02-2}
\end{align}
where
\begin{equation}\label{EQ:div-cul:varphi-up-2}
\varphi'_{\textbf{u},p}(\textbf{v})=l'_\textbf{u}(\textbf{v})-
\theta'_p(\textbf{v})+ s_2(Q_h\textbf{u},\textbf{v}).
\end{equation}
\end{theorem}

\begin{proof}
 Let $(\bu; p)$ be the
solution of the model problem (\ref{maxwell2}). It is not hard to see that the
following holds true:
\begin{equation}\nonumber\label{5.2.Nov-28.200-2}
\begin{split}
 \nabla\times((i \omega \varepsilon+\sigma)^{-1} \nabla \times \bu) +
 i \omega \mu \bu + \mu \nabla p  & =
\nabla \times ( i\omega  \varepsilon+\sigma)^{-1} \textbf{j}_e,\qquad
 \mbox{in}\ \Omega,\\
(i \omega \varepsilon+\sigma)^{-1} \nabla \times \bu \times \bn &
=0, \qquad\qquad \mbox{on}\ \Gamma.
\end{split}
\end{equation}
It follows from Lemma \ref{lemma5.1-2} that
\begin{equation*}
\begin{split}
&\sum_{T\in {\cal T}_h}((i \omega \varepsilon+ \sigma)^{-1} \nabla_w\times(Q_h\textbf{u}),\nabla_w\times
\textbf{v})_T+ ( i \omega \mu  Q_0\textbf{u},\bv_0)_T-(\nabla_w \cdot ( \mu\textbf{v}),{\cal
Q}_hp)_T\\
=&(\nabla\times (i \omega \varepsilon+ \sigma)^{-1}  \textbf{j}_e,\textbf{v}_0)
+l'_\textbf{u}(\textbf{v})-\theta'_p(\textbf{v}),
\end{split}
\end{equation*}
for all $\bv\in  \bV_{h,0}^2$, which gives
\begin{equation}\label{4.13-2}
a_2(Q_h\textbf{u},\textbf{v})-b_2(\textbf{v},{\cal
Q}_hp)=(\nabla\times (i \omega \varepsilon+ \sigma)^{-1}  \textbf{j}_e,\textbf{v}_0)+l'_\textbf{u}(\textbf{v})-
\theta'_p(\textbf{v})+s_2(Q_h\textbf{u},\textbf{v}).
\end{equation}
Subtracting (\ref{3.5}) from (\ref{4.13-2}) gives rise to the first
error equation (\ref{EQ:div-curl:error-eq-01-2}).

Next, from the second equation in (\ref{maxwell2}) and the
commutative relation (\ref{4.4-2}), we have for any $q\in W_h^2$,
 \begin{equation}\label{4.14-2}
0=\sum_{T\in {\cal T}_h}(\nabla\cdot ( \mu\textbf{u}),q)_T=\sum_{T\in {\cal T}_h}({\cal Q}_h(\nabla\cdot
( \mu\textbf{u})),q)_T=\sum_{T\in {\cal T}_h}(\nabla_w\cdot( \mu  Q_h\textbf{u}),q)_T.
\end{equation}
The difference of (\ref{4.14-2}) and (\ref{3.6}) yields the
second error equation (\ref{EQ:div-curl:error-eq-02-2}). This
completes the proof.
\end{proof}

\section{Error Analysis} \label{Section:sec8}The goal of this section is to
derive some error estimates for the numerical approximations $\bE_h$
and $\bH_h$ arising from the weak Galerkin algorithms
\ref{algo1}-\ref{algo2} for the time-harmonic Maxwell equations.
Recall that the error functions, denoted by $\be_h$ and
$\epsilon_h$, are defined as the difference of the numerical
approximation and the $L^2$ projection of the exact solution. The
error equations as presented in Theorems
\ref{Thm:div-curl:theorem-error-eqns}-\ref{Thm:div-curl:theorem-error-eqns-2}
play an important role in the convergence analysis.

\subsection{Some technical inequalities}

Assume that the finite element partition ${\cal T}_h$ of $\Omega$ is
shape regular in the sense as detailed in \cite{wy1202}. Let $T\in
{\cal T}_h$ be an element with $e$ as an edge/face. It is known that
the following trace inequality holds true
\begin{equation}\label{A4}
\|\psi\|_e^2\leqC\big(h_T^{-1}\|\psi\|_T^2+h_T\|\nabla
\psi\|_T^2\big),\qquad \forall \ \psi\in H^1(T).
\end{equation}
For polynomial functions, we have the following inverse inequality
\begin{equation}\label{Div-Curl:inverse}
\|\nabla \phi\|_T \leqC h_T^{-1} \|\phi\|_T.
\end{equation}
In particular, by combining (\ref{A4}) with
(\ref{Div-Curl:inverse}), we arrive at
\begin{equation}\label{Aa-trace}
\|\phi\|_e^2\leqC h_T^{-1}\|\phi\|_T^2
\end{equation}
for any polynomial $\phi$ on $T$ with degree no more than a
prescribed number.

\begin{lemma}\label{lemmaA1}  \cite{wy1202} Let $k\ge 1$ be the order of the WG finite elements,
and $1 \leq r \leq k$. Let $\textbf{w}\in [H^{r+1}(\Omega)]^d$,
$\rho \in H^r (\Omega)$, and $0 \leq m \leq 1$. There holds
\begin{align}\label{A1}
\sum_{T\in{\cal
T}_h}h_T^{2m}\|\textbf{w}-Q_0\textbf{w}\|^2_{T,m}&\leqC
h^{2(r+1)}\|\textbf{w}\|^2_{r+1},\\
\sum_{T\in{\cal T}_h}h_T^{2m}\|\nabla\times
\textbf{w}-\textbf{Q}_h(\nabla\times\textbf{w})\|^2_{T,m}&\leqC
h^{2r}\|\textbf{w}\|^2_{r+1},\label{A2}\\
\sum_{T\in{\cal T}_h}h_T^{2m}\|\rho-{\cal Q}_h\rho\|^2_{T,m}&\leqC
h^{2r}\|\rho\|^2_{r}.\label{A3}
\end{align}
\end{lemma}

For convenience, we introduce two semi-norms in the WG finite element
space $\bV_h$; i.e.,
\begin{equation*}\label{Norm-v1h}
\begin{split}
|\bv|_{1,h} = &\left(\sum_{T\in\T_h}
h_T^{-1}\|(\bv_0-\bv_b)\times\bn\|_{\pT}^2 +
h_T^{-1}\|( \varepsilon\bv_0-\bv_b)\cdot\bn\|_{\pT}^2\right)^{\frac12},\\
|\bv|_{2,h} =& \left(\sum_{T\in\T_h}
h_T^{-1}\|(\bv_0-\bv_b)\times\bn\|_{\pT}^2 +
h_T^{-1}\|( \mu\bv_0-\bv_b)\cdot\bn\|_{\pT}^2\right)^{\frac12}.\\
\end{split}
\end{equation*}

\begin{lemma}\label{Lemma:Div-Curl:phi-estimate}\cite{cwang1}
Assume that the finite element partition ${\cal T}_h$ of $\Omega$ is
shape regular and $1\leq r \leq k$. Let $\bw \in  [H^{r+1}
(\Omega)]^d$ and $p \in H^r (\Omega)$. Then, we have
\begin{align*}
|s_1(Q_h\bw,\bv)| &\leqC h^r\|\bw\|_{r+1} \ |\bv|_{1,h},\\
|s_2(Q_h\bw,\bv)| &\leqC h^r\|\bw\|_{r+1} \ |\bv|_{2,h},\\
|l_\bw(\bv)| & \leqC h^r\|\bw\|_{r+1}\ |\bv|_{1,h},\\
|l'_\bw(\bv)| & \leqC h^r\|\bw\|_{r+1}\ |\bv|_{2,h},\\
|\theta_p(\bv)| & \leqC h^r\|p\|_r\ |\bv|_{1,h},\\
|\theta'_p(\bv)| & \leqC h^r\|p\|_r\ |\bv|_{2,h},
\end{align*}
for any $\bv \in \bV_h$. Here, $l_\bw(\cdot)$, $\theta_p(\cdot)$ and
$l'_\bw(\cdot)$, $\theta'_p(\cdot)$ are defined in
(\ref{l})-(\ref{theta}) and (\ref{l-2})-(\ref{theta-2}),
respectively.
\end{lemma}

\subsection{Error estimates} We are now in a position to present
some error estimates for the weak Galerkin algorithms
\ref{algo1}-\ref{algo2}.

\begin{theorem}\label{Thm:div-curl:theorem-error-estimate}
Assume that $k \geq 1$ is the order of the WG finite elements for
(\ref{3.3})-(\ref{3.4}). Let $(\bE; p)\in [H^{k+1}(\Omega)]^d\times
H^k(\Omega)$ be the solution of the problem (\ref{maxwell1}) and
$(\bE_h; p_h)\in \bV_{h,0}^1 \times W^1_h$ be the WG finite element
solution arising from (\ref{3.3})-(\ref{3.4}). Then, we have the
following estimate
\begin{equation}\label{th1}
\3bar Q_h\textbf{E}-\textbf{E}_h\3bar_{\bV_{h, 0}^1} +\|{\cal
Q}_hp-p_h\|_{W_h^1}\leqC h^k(\|\bE\|_{k+1}+\|p\|_k).
\end{equation}
\end{theorem}

\begin{proof} Theorem
\ref{Thm:div-curl:theorem-error-eqns} implies that the error
functions $\be_h=Q_h\bE - \bE_h$ and $\epsilon_h={\cal Q}_h p - p_h$
satisfy the error equations
(\ref{EQ:div-curl:error-eq-01})-(\ref{EQ:div-curl:error-eq-02}). By
letting $\bv=\be_h$ in (\ref{EQ:div-curl:error-eq-01}) and then
using (\ref{EQ:div-curl:error-eq-02}) we obtain
\begin{equation}\label{November-29:001}
a_1(\be_h, \be_h) = \varphi_{\bE, p}(\be_h).
\end{equation}
The right-hand side of (\ref{November-29:001})  can be estimated by using Lemma
\ref{Lemma:Div-Curl:phi-estimate} as follows
$$
|\varphi_{\bE,p}(\textbf{\be}_h)|\leqC
h^k(\|\bE\|_{k+1}+\|p\|_k)|\textbf{\be}_h|_{1,h}.
$$
Substituting the above into (\ref{November-29:001}) yields
\begin{equation*}
|a_1(\be_h, \be_h)| \leqC
h^k(\|\bE\|_{k+1}+\|p\|_k)|\textbf{\be}_h|_{1,h},
\end{equation*}
which, together with the coercivity $|\textbf{\be}_h|_{1,h}^2
\lesssim |a_1(\be_h, \be_h)|$, leads to
\begin{equation}\label{November-29:002}
|a_1(\be_h, \be_h)|^{1/2} \leqC h^k(\|\bE\|_{k+1}+\|p\|_k).
\end{equation}
From the imaginary part of $a_1(\be_h, \be_h) $ and (\ref{November-29:002}),  we have
$$
\Big(\sum_{T\in {\cal T}_h} \|\be_0\|^2_T \Big)^{1/2} \leqC
h^k(\|\bE\|_{k+1}+\|p\|_k),
$$
which, combining with the real part of $a_1(\be_h, \be_h) $ and (\ref{November-29:002}), yields
\begin{equation*}
 \begin{split}
 \Big(\sum_{T\in {\cal T}_h}\|\nabla_w \times\be_h\|^2_T+& h_T^{-1}\|(\varepsilon \be_0-\be_b) \cdot \bn\|_{\partial T}^2+h_T^{-1}\|(  \be_0-\be_b)  \times \bn\|_{\partial T}^2 \Big)^{1/2}
 \\ \leqC  & h^k(\|\bE\|_{k+1}+\|p\|_k).
 \end{split}
\end{equation*}
Thus, we have
$$
\3bar\be_h\3bar_{\bV_{h, 0}^1} \leqC h^k(\|\bE\|_{k+1}+\|p\|_k).
$$

The error function $\epsilon_h$ can be estimated by using the {\em
inf-sup} condition derived in Lemma \ref{lvq2}. To this end, from
the equation (\ref{EQ:div-curl:error-eq-01}), we have
\begin{equation}\label{November:30:finenow}
b_1(\bv,\epsilon_h) = -\varphi_{\bE,p}(\bv) + a_1(\be_h,\bv).
\end{equation}
By using Lemma \ref{lvq2} and letting $\bv = \bv_{\epsilon_h}$ in
(\ref{November:30:finenow}) we arrive at
$$
\|\epsilon_h\|_{W_h^1}^2 \leqC |\varphi_{\bE,p}(\bv_{\epsilon_h})| +
|a_1(\be_h,\bv_{\epsilon_h})|.
$$
It now follows from Lemma \ref{Lemma:Div-Curl:phi-estimate} and the
error estimate (\ref{November-29:002}) that
$$
\|\epsilon_h\|_{W_h^1}^2 \leqC h^k (\|\bE\|_{k+1}+\|p\|_k) \3bar
\bv_{\epsilon_h}\3bar_{\bV_{h, 0}^1} ,
$$
which, together with (\ref{6.5}), leads to
$$
\|\epsilon_h\|_{W_h^1} \leqC h^k (\|\bE\|_{k+1}+\|p\|_k).
$$
This completes the proof of the theorem.
\end{proof}

\begin{theorem}\label{Thm:div-curl:theorem-error-estimate-2}
Assume that $k \geq 1$ is the order of the WG finite elements
employed in the scheme (\ref{3.5})-(\ref{3.6}). Let $(\bH; p)\in
[H^{k+1}(\Omega)]^d\times H^k(\Omega)$ be the solution of the
problem (\ref{maxwell2}) and $(\bH_h; p_h)\in \bV_{h,0}^2 \times
W^2_h$ be the WG finite element solution arising from
(\ref{3.5})-(\ref{3.6}). Then, we have
\begin{equation}\label{th1-2}
\3bar Q_h\textbf{H}-\textbf{H}_h\3bar_{\bV_{h, 0}^2} +\|{\cal
Q}_hp-p_h\|_{W_h^2}\leqC h^k(\|\bH\|_{k+1}+\|p\|_k).
\end{equation}
\end{theorem}

\begin{proof} From Theorem \ref{Thm:div-curl:theorem-error-eqns-2} we see that the error functions
$\be_h=Q_h\bH - \bH_h$ and $\epsilon_h={\cal Q}_h p - p_h$ satisfy
the error equations
(\ref{EQ:div-curl:error-eq-01-2})-(\ref{EQ:div-curl:error-eq-02-2}).
By setting $\bv=\be_h$ in (\ref{EQ:div-curl:error-eq-01-2}) and then
using (\ref{EQ:div-curl:error-eq-02-2}) we obtain
\begin{equation}\label{November-29:001-2}
a_2(\be_h, \be_h) = \varphi'_{\bH, p}(\be_h).
\end{equation}
The right-hand side  of (\ref{November-29:001-2}) can be handled by
using Lemma \ref{Lemma:Div-Curl:phi-estimate} as follows
$$
|\varphi'_{\bH,p}(\textbf{\be}_h)|\leqC
h^k(\|\bH\|_{k+1}+\|p\|_k)|\textbf{\be}_h|_{2,h}.
$$
Substituting the above estimate into (\ref{November-29:001-2})
yields
\begin{equation*}
|a_2(\be_h, \be_h)| \leqC
h^k(\|\bH\|_{k+1}+\|p\|_k)|\textbf{\be}_h|_{2,h},
\end{equation*}
which, together with the coercivity estimate
$|\textbf{\be}_h|_{2,h}^2 \lesssim |a_2(\be_h, \be_h)|$, leads to
\begin{equation}\label{November-29:002-2}
|a_2(\be_h, \be_h)|^{1/2} \leqC h^k(\|\bH\|_{k+1}+\|p\|_k),
\end{equation}
From the real part of $a_2(\be_h, \be_h) $ and
(\ref{November-29:002-2}), we obtain
\begin{equation}\label{ima1}\begin{split}
 \Big(\sum_{T\in {\cal T}_h}\|\nabla_w\times\be_h\|^2_T \Big)^{1/2} &\leqC  h^k(\|\bH\|_{k+1}+\|p\|_k),\\
\Big(\sum_{T\in {\cal T}_h} h_T^{-1}\|(\mu \be_0-\be_b) \cdot
\bn\|_{\partial T}^2 \Big)^{1/2} &\leqC
h^k(\|\bH\|_{k+1}+\|p\|_k),\\
\Big(\sum_{T\in {\cal T}_h} h_T^{-1}\|( \be_0-\be_b)  \times
\bn\|_{\partial T}^2 \Big)^{1/2} &\leqC h^k(\|\bH\|_{k+1}+\|p\|_k).
\end{split}
\end{equation}
Combining with the imaginary part of $a_2(\be_h, \be_h) $,
(\ref{November-29:002-2}) and (\ref{ima1}) gives rise to
$$
\Big(\sum_{T\in {\cal T}_h} \|\be_0\|^2_T \Big)^{1/2}
 \leqC  h^k(\|\bH\|_{k+1}+\|p\|_k).
$$
Thus,
$$
\3bar\be_h\3bar_{\bV_{h, 0}^2}   \leqC  h^k(\|\bH\|_{k+1}+\|p\|_k).
$$

The error function $\epsilon_h$ can be estimated by using the {\em
inf-sup} condition derived in Lemma \ref{lvq2-2}. To this end, from
the equation (\ref{EQ:div-curl:error-eq-01-2}), we have
\begin{equation}\label{November:30:finenow-2}
b_2(\bv,\epsilon_h) = -\varphi'_{\bH,p}(\bv)+ a_2(\be_h,\bv).
\end{equation}
Using Lemma \ref{lvq2-2} and  letting $\bv = \bv_{\epsilon_h}$ in
(\ref{November:30:finenow-2}) yields
$$
\|\epsilon_h\|_{W_h^2}^2 \leqC |\varphi'_{\bH,p}(\bv_{\epsilon_h})|
+ |a_2(\be_h,\bv_{\epsilon_h})|.
$$
It now follows from Lemma \ref{Lemma:Div-Curl:phi-estimate} and the
error estimate (\ref{November-29:002-2}) that
$$
\|\epsilon_h\|_{W_h^2}^2 \leqC h^k (\|\bH\|_{k+1}+\|p\|_k) \3bar
\bv_{\epsilon_h}\3bar_{\bV_{h, 0}^2},
$$
which, together with (\ref{6.5-2}), leads to
$$
\|\epsilon_h\|_{W_h^2} \leqC h^k (\|\bH\|_{k+1}+\|p\|_k).
$$
This completes the proof of the theorem.
\end{proof}

\subsection{$L^2$-error estimates}
In this subsection, we shall present a $L^2$-error estimate for the
components $\be_0$ and $\be_b$ in the error function $\be_h$ for the
WG algorithms \ref{algo1} and \ref{algo2}. To this end, let us
introduce a $L^2$-like norm for the edge/face component $\bv_b$ in
the weak function $\bv=\{\bv_0; \bv_b\}\in \bV_h$ as follows:
$$
\|\bv_b\|_{\E_h}=\Big(\sum_{T\in {T}_h} h_T \int_{\partial T}
|\bv_b|^2 ds\Big)^{\frac{1}{2}}.
$$

\begin{theorem}\label{THM:L2-Electric}
Let $k\geq 1$ be the order of the WG finite element employed in the
scheme (\ref{3.3})-(\ref{3.4}). Let $(\textbf{E}; p) \in
[H^{k+1}(\Omega)]^d \times H^{k}(\Omega)$ and $(\textbf{E}_h; p_h)
\in \bV_{h,0}^1 \times W^1_h$ be the solutions of the problem
(\ref{maxwell1}) and (\ref{3.3})-(\ref{3.4}), respectively. Then,
the following estimate holds true:
\begin{equation*}
\|Q_0\textbf{E}-\textbf{E}_0\|\leqC
h^{k+1}\big(\|\textbf{E}\|_{k+1}+\|p\|_k\big),
\end{equation*}
\begin{equation*}
\|\mathds{Q}_b\textbf{E}-\textbf{E}_b\|_{\E_h}\leqC
h^{k+1}\big(\|\textbf{E}\|_{k+1}+\|p\|_k\big).
\end{equation*}
\end{theorem}

Likewise, for the magnetic field intensity approximation, we have
the following result.

\begin{theorem}\label{THM:L2-Magnetic}
Let $k\geq 1$ be the order of the WG finite elements employed in the
WG scheme (\ref{3.5})-(\ref{3.6}). Let $(\textbf{H}; p) \in
[H^{k+1}(\Omega)]^d \times H^{k}(\Omega)$ and $(\textbf{H}_h; p_h)
\in \bV_{h,0}^2 \times W^2_h$ be the solutions of the problem
(\ref{maxwell2}) and (\ref{3.5})-(\ref{3.6}), respectively. Then the
following $L^2$-error estimates hold true:
\begin{equation*}
\|Q_0\textbf{H}-\textbf{H}_0\|\leqC
h^{k+1}\big(\|\textbf{H}\|_{k+1}+\|p\|_k\big),
\end{equation*}
\begin{equation*}
\|\mathds{Q}_b\textbf{H}-\textbf{H}_b\|_{\E_h}\leqC
h^{k+1}\big(\|\textbf{H}\|_{k+1}+\|p\|_k\big).
\end{equation*}
\end{theorem}

A proof for Theorems \ref{THM:L2-Electric} and \ref{THM:L2-Magnetic}
can be given by following a routine duality argument readily
available in the finite element method. Readers are referred to
\cite{cwang1} for more details on a model problem that resembles the
time harmonic Maxwell equations.

\end{document}